\documentclass[12pt]{amsart}
\usepackage{amscd,amsmath,amsthm,amssymb}
\usepackage{color}
\usepackage{amsfonts,amssymb,amscd,amsmath,enumerate,verbatim, bm}
\usepackage{lineno}
 \def\G{{\mathcal G}}

 \def\opn#1#2{\def#1{\operatorname{#2}}} 
 \opn\chara{char} \opn\length{\ell} \opn\pd{pd} \opn\rk{rk}
 \opn\projdim{proj\,dim} \opn\injdim{inj\,dim} \opn\rank{rank}
 \opn\depth{depth} \opn\grade{grade} \opn\height{height}
 \opn\embdim{emb\,dim} \opn\codim{codim}
 
 \opn\Tr{Tr} \opn\bigrank{big\,rank}
 \opn\superheight{superheight}\opn\lcm{lcm}
 \opn\trdeg{tr\,deg}
 \opn\reg{reg} \opn\lreg{lreg} \opn\ini{in} \opn\lpd{lpd}
 \opn\size{size} \opn\sdepth{sdepth}
 \opn\link{link}\opn\fdepth{fdepth}\opn\lex{lex}

\def\Gr{\operatorname{Gr\ddot{o}b}} 

\def\LT{\operatorname{LT}}

 \opn\div{div} \opn\Div{Div} \opn\cl{cl} \opn\Cl{Cl}
 \opn\Spec{Spec} \opn\Supp{Supp} \opn\supp{supp} \opn\Sing{Sing}
 \opn\Ass{Ass} \opn\Min{Min}\opn\Mon{Mon}
 \opn\Ann{Ann} \opn\Rad{Rad} \opn\Soc{Soc}
 \opn\Im{Im} \opn\Ker{Ker} \opn\Coker{Coker} \opn\Am{Am}
 \opn\Hom{Hom} \opn\Tor{Tor} \opn\Ext{Ext} \opn\End{End}
 \opn\Aut{Aut} \opn\id{id}
 
 \opn\nat{nat}
 \opn\pff{pf}
 \opn\Pf{Pf} \opn\GL{GL} \opn\SL{SL} \opn\mod{mod} \opn\ord{ord}
 \opn\Gin{Gin} \opn\Hilb{Hilb}\opn\sort{sort}
 \opn\aff{aff} \opn
 \con{conv} \opn\relint{relint} \opn\st{st}
 \opn\lk{lk} \opn\cn{cn} \opn\core{core} \opn\vol{vol}  \opn\inp{inp} \opn\nilpot{nilpot}
 \opn\link{link} \opn\star{star}\opn\lex{lex}\opn\set{set}
 \opn\width{wd}
 \opn\ecart{ecart}
 \opn\gr{gr}
 
 \def\pot#1#2{#1[\kern-0.28ex[#2]\kern-0.28ex]}

 \opn\dirlim{\underrightarrow{\lim}}
 \opn\inivlim{\underleftarrow{\lim}}
 \let\sect=\cap

 \def\Implies{\ifmmode\Longrightarrow \else
         \unskip${}\Longrightarrow{}$\ignorespaces\fi}
 \def\implies{\ifmmode\Rightarrow \else
         \unskip${}\Rightarrow{}$\ignorespaces\fi}
 \def\iff{\ifmmode\Longleftrightarrow \else
         \unskip${}\Longleftrightarrow{}$\ignorespaces\fi}

 \let\:=\colon
 \newtheorem{Theorem}{Theorem}[section]
 \newtheorem{Lemma}[Theorem]{Lemma}
 \newtheorem{Corollary}[Theorem]{Corollary}
 \newtheorem{Proposition}[Theorem]{Proposition}
 \newtheorem{Remark}[Theorem]{Remark}
 
 \newtheorem{Example}[Theorem]{Example}
 
 \newtheorem{Definition}[Theorem]{Definition}

 \let\epsilon\varepsilon
 \let\kappa=\varkappa
 \def\qed{\ifhmode\textqed\fi
       \ifmmode\ifinner\quad\qedsymbol\else\dispqed\fi\fi}
 \def\textqed{\unskip\nobreak\penalty50
        \hskip2em\hbox{}\nobreak\hfil\qedsymbol
        \parfillskip=0pt \finalhyphendemerits=0}
 \def\dispqed{\rlap{\qquad\qedsymbol}}
 
 \opn\dis{dis}
 \def\pnt{{\raise0.5mm\hbox{\large\bf.}}}
 
 \opn\Lex{Lex}

\begin{document}

\title {Gr\"{o}bner-nice pairs of ideals}

\author {Mircea Cimpoea\c s}
\address{Mircea Cimpoea\c s, Simion Stoilow Institute of Mathematics of the Romanian Academy, 
Research unit 5,  P.O.Box 1-764,  Bucharest   014700, Romania }
\email{mircea.cimpoeas@imar.ro}
 
\author {Dumitru I. Stamate}
\address{Dumitru I. Stamate, Faculty of Mathematics and Computer Science, University of Bucharest, Str. Academiei 14, Bucharest, Romania, and  \newline  \indent
Simion Stoilow Institute of Mathematics of the Romanian Academy, Research group
of the project ID-PCE-2011-1023, P.O.Box 1-764, Bucharest 014700, Romania} 

\email{dumitru.stamate@fmi.unibuc.ro}

\begin{abstract}
We introduce the concept of a Gr\"{o}bner nice pair of ideals in a polynomial ring and we present some  applications.
\end{abstract}

\thanks{}
 


\maketitle

\section*{Introduction}

One feature of a Gr\"obner basis is that it extends a system of generators for an ideal in a polynomial ring so that  several invariants or algebraic properties are easier to read.
It is a natural question to ask how to obtain a Gr\"obner basis for an ideal obtained by performing  basic algebraic operations.
A first situation we discuss in this note is when a Gr\"obner basis for the sum of the ideals $I$ and $J$ is obtained by taking the union of two Gr\"obner bases for the respective ideals. In that case we say that $(I,J)$ is a Gr\"obner nice pair ($G$-nice pair, for short). In Theorem~\ref{thm:main} we prove that for any given monomial order, $(I,J)$ is a $G$-nice pair if and only if $\ini(I+J)=\ini(I)+ \ini(J)$, which is also equivalent to having $\ini(I\cap J)=\ini(I) \cap \ini(J)$. 

Given the ideals $I$ and $J$ in a polynomial ring, they could be a $G$-nice pair for some, for any, or for no monomial order. Situations of $G$-nice pairs of ideals have naturally occurred in the literature, especially related to ideals of minors, e.g. \cite{conca, HoSu, MohRauh}.

One application of the new concept is in Corollary~\ref{many-regular}: assume  $J$ is any ideal in the polynomial ring $S$  and let $f_1,\ldots,f_r$ in $S$ be a regular sequence on $S/J$. Then the sequence $\ini(f_1),\ldots,\ini(f_r)$ is regular on $S/\ini(J)$
if and only if $
\ini(J,f_1,\ldots,f_r)=\ini(J)+(\ini(f_1),\ldots,\ini(f_r))
$. 

We introduced the notion of a Gr\"obner nice pair in an attempt to unify some cases of distributivity in the lattice of ideals in a polynomial ring $S$. For instance, one consequence of Proposition~\ref{prop:distributivity} is that if $(J,E), (J,E')$ and $(J, E\cap E')$ are $G$-nice pairs, then $(J+E)\cap (J+E')=J+ (E\cap E')$ if and only if $\ini((J+E)\cap (J+E'))= \ini(J) +\ini(E\cap E')$.
Moreover, in Proposition~\ref{intersections-monomials} we show that when $(E_i)_{i\in \Lambda}$ is a family of monomial ideals  such that $(J,E_i)$ is $G$-nice for all $i\in \Lambda$, then $(J, \bigcap_{i\in \Lambda} E_i)$ is a $G$-nice pair and $\bigcap_{i\in \Lambda} (J+E_i)= J+ (\bigcap_{i \in \Lambda} E_i)$.

One problem we raise in Section~\ref{sec2} is how to efficiently transform a pair of ideals $(J,E)$ into a $G$-nice pair $(J,F)$ so that $F\supseteq E$. We show that in general there is no minimal such ideal $F$  so that also $J+F=J+E$, see Example~\ref{exm:infinite}. However, if $E$ is a monomial ideal then we may consider $\widehat{E}$ the smallest monomial ideal in $S$ containing $E$ and so that $(J,\widehat{E})$ is $G$-nice. Furthermore, when $J$ is a binomial ideal (i.e. it is generated by binomials) then $J+E=J+\widehat{E}$, see Corollary~\ref{binomial-hat}.

Buchberger's criterion (\cite{EH, Eis, sing-book}) asserts that a set of polynomials form a Gr\"obner basis for the ideal they generate if and only if for any two elements in the set their $S$-polynomial reduces to zero with respect to the given set. Having this fact in mind, a special class of Gr\"obner nice pairs is introduced in Section~\ref{sec3}. Namely, if $\G_J$ is a Gr\"obner basis for $J$, then we say that the ideal $E$ is $S$-nice with respect to $\G_J$ if for all $f\in \G_J$ and $g\in E$ their $S$-polynomial $S(f,g)\in E$. 
To check that property it is enough to verify it for $g$ in some Gr\"obner basis for $E$, see Proposition~\ref{snice}. A good property is that if $E_i$ is $S$-nice with respect to $\G_J$ for all $i\in \Lambda$, then so is $\bigcap_{i\in \Lambda} E_i$.

This way, given $\G_J$, for any ideal (resp. monomial ideal $E$) we can define the $S$-nice (monomial) closure: $\widetilde{E}$ (resp. $E^\sharp$) is the smallest (monomial) ideal which is $S$-nice w.r.t. $\G_J$ and we show  how to compute it in Proposition~\ref{inductive}.
While $J+\widetilde{E}=J+E$, it is not always the case that $J+E^\sharp=J+E$.  In Proposition~\ref{proplast} we show that if $\G_J$ consists of binomials and $E$ is a monomial ideal, then $\widehat{E}=\widetilde{E}$.

We provide many examples for the notions that we introduce.

\section{Gr\"obner-nice pairs of ideals}
\label{sec:gnice}
	Let $K$ be any field. Throughout this paper, we usually denote by $S$ the polynomial ring $K[x_1,\ldots,x_n]$, unless it is stated otherwise. For an ideal $I$ in $S$ and a monomial order $\leq$ on $S$,  the set of all Gr\"obner bases of $I$ with respect to the given monomial order will be denoted $\Gr_\leq (I)$, or simply $\Gr(I)$ when there is no risk of confusion.  As a piece of notation, with respect to a fixed monomial order,  for $f$ in $S$ its leading term is denoted $\LT(f)$, and its leading monomial $\ini(f)$.
	
 The following result is at the core of our work.

\begin{Theorem}
\label{thm:main}
We fix a monomial order in the polynomial ring $S$.
Let $J$ and $E$ be ideals in $S$.  The following conditions are equivalent:
\begin{enumerate}
\item[(a)] $\ini(J+E)=\ini(J)+\ini(E)$;
\item[(b)] for any $\mathcal G_J \in \Gr(J)$ and $\mathcal G_E \in \Gr(E)$, we have $\mathcal G_J \cup \mathcal G_E \in \Gr(J+E)$;
\item[(c)] there exist   $\mathcal G_J \in \Gr(J)$ and $\mathcal G_E \in \Gr(E)$, such that $\mathcal G_J \cup \mathcal G_E \in \Gr(J+E)$;
\item[(d)] $\ini(J\cap E) = \ini(J)\cap \ini(E)$;
\item[(e)] for any $f\in J$ and $g\in E$, there exists $h\in J\cap E$ such that $\ini(h)=\lcm(\ini f,\ini g)$;
\item[(f)] for any $0\neq h\in J+E$, there exist $f\in J$ and $g\in E$ with $h=f-g$ and $\ini(f)\neq \ini(g)$.
\end{enumerate}
\end{Theorem}

\begin{proof} We set $I=J+E$.

 (a) \implies (b):  Let $\mathcal G_J = \{f_1,\ldots,f_r\} \in \Gr(J)$ and $\mathcal G_E=\{g_1,\ldots,g_p\}\in \Gr(E)$. 
Clearly,  $\mathcal G_J\cup \mathcal G_E$ generates the ideal $I$. 
Since $\ini(J)=(\ini(f_1),\ldots,\ini(f_r))$ and $\ini(E)=(\ini(g_1),\ldots,\ini(g_p))$,  
by (a), it follows that $\ini(I) = (\ini(f_1),\ldots,\ini(f_r),\ini(g_1),\ldots,\ini(g_p))$ and therefore, $\mathcal G_J\cup \mathcal G_E$ is a Gr\"obner basis of $I$.

(b) \implies (c) is trivial. 

 (c)\implies (a): Let $\mathcal G_J = \{f_1,\ldots,f_r\} \in \Gr(J)$ and $\mathcal G_E=\{g_1,\ldots,g_p\}\in \Gr(E)$, 
such that $\mathcal G_J \cup \mathcal G_E \in \Gr(J+E)$. It follows that $\ini(J+E)=\ini(J)+\ini(E)$, as required.

(a) \implies (d): Since always $\ini(J\cap E) \subseteq \ini(J)\cap \ini(E)$, it remains to prove the other inclusion. 
Let $m$ be any monomial in $\ini(J)\cap \ini(E)$. 
Hence there exist $f\in J$ and $g\in E$ such that $m=\ini(f)=\ini(g)$. If $f=g$ then $m\in \ini(J\cap E)$ and we are done. 

Otherwise, since
$f-g\in J+E$, by (a), $\ini(f-g)\in \ini(J)+\ini(E)$. If $\ini(f-g)\in \ini(J)$, then there exists $f_1\in J$ with $\LT(f-g)=\LT(f_1)$, and we set 
$g_1=0$. If $\ini(f-g)\in \ini(E)$, then there exists $g_1\in E$ with $\LT(f-g)=\LT(g_1)$, and we set 
$f_1=0$. In either case, we have $\ini(f-f_1)=\ini(g-g_1)=m$. If $f-f_1=g-g_1$ we are done. Otherwise, since $\ini((f-f_1)-(g-g_1))<\ini(f-g)$, 
we repeat the above procedure for $f-f_1$ and $g-g_1$. By Dickson's Lemma (\cite[Theorem 1.9]{EH}) this process eventually stops, hence $m\in \ini(J\cap E)$.

Condition (e) is a restatement of (d). 

(e) \implies (a):
  It is enough to prove that $\ini(J+E)\subseteq \ini(J)+\ini(E)$. Let $p\in J+E$ and write $p=f-g$, with $f\in J$ and $g\in E$.
If $\LT(f)\neq \LT(g)$ then $\ini(p)=\ini(f)$ or $\ini(p)=\ini(g)$ and we are done. Assume $\LT(f)=\LT(g)$. By (e), there exists $h\in J\cap E$ with 
$\LT(h)=\LT(f)=\LT(g)$. We let $f_1=f-h$ and $g_1=g-h$. We note that $p=f-g=f_1-g_1$, $\ini(f_1)<\ini(f)$ and $\ini(g_1)<\ini(g)$. If $\LT(f_1)\neq \LT(g_1)$  then $\ini(p)\in \ini(J)+\ini(E)$, arguing as above. Otherwise, we repeat the same procedure for $f_1$ and $g_1$ until it stops.

(f) \implies (a) is obvious. 

(d) \implies (f): Let $h\in I$. We write $h=f-g$ with $f\in J$ and $g\in E$. If $h\in J$ or $h\in E$ then there is nothing to prove. 
 
If $\ini(f)\neq \ini(g)$, we are done. 
Otherwise, assume $\ini(f)=\ini(g)$. 
We distinguish two cases, depending whether $\ini(f) =\ini(h)$ or $\ini(f)>\ini(h)$.
 
If  $\ini(h)=\ini(f)=\ini(g)$,  according to (d) we have that $\ini(h)\in \ini(J)\cap \ini(E) = \ini(J\cap E)$. Therefore, there exists $w\in J\cap E$ with $\LT(w)=\LT(g)$.
We write $h=(f-w)-(g-w)$. Note that $f-w\in J$, $g-w\in E$ and $\ini(g-w)<\ini(g)=\ini(h)$. It follows that $\ini(h)=\ini(f-w)>\ini(g-w)$ so we are done.
Note that $g-w\neq 0$, otherwise  $h$ would be in $J$, a contradiction.

If $\ini(h)< \ini(f)=\ini(g)=:m$, since $\ini(h)=\ini(f-g)<m$, it follows that $\LT(f)=\LT(g)$. Also, $m\in \ini(J)\cap \ini(E)=\ini(J\cap E)$. Let $w\in J\cap E$
with $\ini(w)=m$. We can assume $\LT(w)=\LT(f)=\LT(g)$. We write $h=f-g=f_1-g_1$, where $f_1=f-w$ and $g_1=g-w$. 
Arguing as before, $0\neq f_1 \in J$, $0\neq g_1\in E$, 
$\ini(f_1) < m$ and $\ini(g_1) < m$. If $\ini(f_1)\neq \ini(g_1)$ we are done. If $\ini(f_1)=\ini(g_1)=\ini(h)$ we are done by the discussion of the former case. Otherwise, we apply the same procedure for $f_1$ and $g_1$, until it  eventually stops.
\end{proof}

\begin{Remark}
{\em
The equivalence of the conditions (a), (d) and (e) in Theorem \ref{thm:main} was proved in a different way by A.~Conca in \cite[Lemma 1.3]{conca} under the extra hypothesis that the ideals $J$ and $E$ are homogeneous.
}
\end{Remark}

\begin{Definition}
\label{def:gnice}
{\em
Let $S$ be a polynomial ring with a fixed monomial order.  If the ideals $J$ and $E$ of $S$   fulfill 
one of the equivalent conditions of Theorem \ref{thm:main}, we say that
$(J,E)$ is a {\em Gr\"obner nice (G-nice) pair} of ideals.
}
\end{Definition}

\begin{Remark}
{\em
The chosen monomial order is essential. For example, let $J=(x^2+y^2)$, $E=(x^2)$ and 
    $I=J+E = (x^2,y^2)$ as ideals in $S=K[x,y]$. 
   If $x>y$,  then $\ini(J)=\ini(E)=(x^2)$, but $\ini(I)=(x^2,y^2)$. Thus the pair $(J,E)$ is not $G$-nice.
   On the other hand, if $y>x$,  then $\ini(I)=\ini(J)+\ini(E)=(x^2,y^2)$, hence the pair $(J,E)$ is $G$-nice.

On the other hand, some pairs of ideals are never $G$-nice, regardless of the monomial order which is used.
Indeed, let $J=(x^2+y^2)$ and $E=(xy)$ in $S=K[x,y]$. It is easy to see that $S(x^2+y^2, xy)$ equals either $y^3$ (if  $x^2>y^2$), or $x^3$ (if $x^2<y^2$), and in either case  $\ini (S(x^2+y^2, xy) ) \notin (\ini (J), \ini (E))$.
}
\end{Remark}

Here are some examples of classes of Gr\"obner nice pairs of ideals.
\begin{Example}
{\em
\begin{enumerate}
\item If $J \subseteq E\subseteq S$, then the pair $(J,E)$ is $G$-nice.
\item If $J$ and $E$ are monomial ideals in $S$, then the pair $(J,E)$ is $G$-nice.
\item If $J$ and $E$ are ideals in $S$ whose generators involve disjoint sets of variables, then the pair $(J,E)$ is $G$-nice.
\item (Conca \cite{conca}) Let $X=(x_{ij})$ be an $n\times m$ matrix of indeterminates and 
let $Z$ be a set of consecutive rows (or columns) of $X$. For  $t$  an integer with $1\leq t \leq \min\{n,m\}$
we let $J=I_t(X)$ be the ideal in $S=K[X]$ generated by the $t$-minors of $X$. Also, let 
$E=I_{t-1}(Z)\subset S$. Then $(J,E)$ is a $G$-nice pair of ideals, according to \cite[Proposition 3.2]{conca}.

More generally, \cite[Proposition 3.3]{conca} states that when $Y$ is a {\em ladder}, $J=I_t(X)$ and $E=I_{t-1}(Y\cap Z)$,
then the pair $(J,E)$ is $G$-nice. We refer to \cite{conca} for the unexplained terminology and 
further details on this topic.
\item Ideals generated by various minors in a generic matrix are a source of $G$-nice pairs of ideals, see \cite[Lemma 4.2]{HoSu}, \cite[Lemma 2.10]{MohRauh}.
\end{enumerate}
}
\end{Example}

The following results characterize situations when one of the ideals of the $G$-pair is generated by a regular sequence.

\begin{Proposition}
\label{prop:regular}
Let $J$ be an ideal in  the polynomial ring $S$ and let $f\in S$ which is regular on $S/J$. 
The following conditions are equivalent:
\begin{enumerate}
\item[(a)] $\ini(J,f)=\ini(J)+(\ini(f))$;
\item[(b)] $\ini(f)$ is regular on $S/\ini(J)$.
\end{enumerate}
\end{Proposition}

\begin{proof}
We note that since $f$ is regular on $S/J$ we get that 
$$\ini(J\cap(f))=\ini(fJ)=\ini(f)\ini(J).$$
 By Theorem \ref{thm:main}(d), property (a) is equivalent to $\ini(J\cap(f))=\ini(J)\cap (\ini(f))$. 
That in turn is equivalent to $\ini(f)\ini(J)=\ini(J)\cap (\ini(f))$, which is a restatement of the condition (b), since $S$ is a domain.
\end{proof}

\begin{Corollary}
\label{many-regular}
Let $J$ be any ideal in the polynomial ring $S$  and let $f_1,\ldots,f_r$ in $S$ be a regular sequence on $S/J$. Then the sequence $\ini(f_1),\ldots,\ini(f_r)$ is regular on $S/\ini(J)$
if and only if 
$$
\ini(J,f_1,\ldots,f_r)=\ini(J)+(\ini(f_1),\ldots,\ini(f_r)).
$$

In particular, if $f_1,\ldots,f_r$ is a regular sequence on $S$, then $\ini(f_1),\ldots,\ini(f_r)$ is regular on $S$ if and only if 
$\{f_1,\ldots,f_r\}$ is a Gr\"obner basis for $(f_1,\ldots,f_r)$.
\end{Corollary}

\begin{proof}
This follows from Proposition \ref{prop:regular} by induction on $r$.
\end{proof}

The $G$-nice condition  is also connected to the distributivity property in the lattice of ideals of $S$, as the following result shows.

\begin{Proposition}
\label{prop:distributivity}
Let $J,E$ and $E'$ be ideals in the polynomial ring $S$ such that $(J,E)$ and $(J,E')$ are $G$-nice pairs. 
The following conditions are equivalent:
\begin{enumerate}
\item[(a)] $(J+E)\cap (J+E')=J + (E\cap E')$ and $(J,E\cap E')$ is a $G$-nice pair of ideals;
\item[(b)] $\ini((J+E)\cap (J+E')) = \ini(J) + \ini(E\cap E')$.
\end{enumerate}
\end{Proposition}

\begin{proof}
(a) \implies (b) is straightforward.

(b) \implies (a): We denote   $I =J+E$ and $I'=J+E'$. We have that 
$$
\ini(J)+\ini(E\cap E') \subseteq \ini(J+(E\cap E')) \subseteq \ini(I\cap I')
$$ 
and thus, by (b), these inclusions are in fact equalities. In particular, $\ini(J)+\ini(E\cap E') = \ini(J+(E\cap E'))$, hence  the pair $(J,E\cap E')$ is $G$-nice. 

On the other hand, since $J+(E\cap E')\subseteq I\cap I'$ and $\ini(J+(E\cap E')) = \ini(I\cap I')$, it follows that $I\cap I'=J + (E\cap E')$.
\end{proof}

The two parts  of condition (a) in Proposition \ref{prop:distributivity} are independent, as the following example shows.

\begin{Example} \label{exm}
{\em 
\begin{enumerate}
\item 
Let $J=(x^2+y^2+z^2)$, $E=(xy,y^3+yz^2)$ and $E'=(xy,y^3+yz^2+x^2+y^2+z^2)$ be ideals in $S=K[x,y,z]$.
Then  $I=J+E=J+E'=(J,xy)$.

 On $S$ we consider the reverse lexicographic  monomial order (or revlex, for short)  with $x>y>z$.
 We have $\ini(I)=(x^2,xy,y^3)$, $\ini(E)=\ini(E')=(xy, y^3)$ and one can check with Singular (\cite{Sing}) that 
$\ini(E\cap E')=(xy,y^4)$.

Therefore, $(J,E)$ and $(J,E')$ are $G$-nice pairs of ideals, $(J+E)\cap(J+E')=J+ (E\cap E') = I$,
 but the pair $(J,E\cap E')$ is not $G$-nice.

\item  In $S=K[x,y,z,t]$ let $J=(x^4 + y^3 + z^2, xy^3 - t^2)$, 
$E=(-y^2zt^2 + xz^2 + t^2, x^2yz^2 - zt^4 + xyt^2, x^2zt^4 + y^4z^2 - x^3yt^2 + yz^4, yzt^6 - xy^2t^4 - x^3z^3 - x^2zt^2, -zt^8 - y^5z^3 + xyt^6 - y^2z^5 - y^3z^2 + x^3t^2 - z^4)$
 and $E'=(xz^5 - yzt^2,  y^4zt^2 - z^5t^2, y^3z^5 + z^7 + x^3yzt^2, y^2z^5t^2 + y^3z^3t^2 + x^3zt^4, -z^9t^2 - yz^7t^2 - x^3y^2zt^4)$.  

We set $I=J+E$ and $I'=J+E'$. 

We claim that the inclusion $J+(E\cap E')\subset I\cap I'$  is strict. 
Indeed, considering the   reverse lexicographic order on $S$ with $x>y>z>t$, one can check with Singular (\cite{Sing}) that  $y^2z^6t^2 \in \ini(I\cap I')$, but $y^2z^6t^2 \notin \ini(J + (E\cap E'))$. However, one can verify that the pair $(J,E\cap E')$ is $G$-nice.
\end{enumerate}
}
\end{Example}

The following result  is a dual form of Proposition \ref{prop:distributivity}.

\begin{Proposition}
Let $J,E$ and $E'$ be   ideals in the polynomial ring $S$ such that the pairs $(J,E)$ and $(J,E')$ are $G$-nice. 
The following conditions are equivalent:
\begin{enumerate}
\item[(a)] $J \cap E + J \cap E'=J \cap (E + E')$ and the pair $(J,E+E')$ is $G$-nice.
\item[(b)] $\ini(J \cap E + J \cap E') = \ini(J) \cap \ini(E + E')$.
\end{enumerate}
\end{Proposition}

\begin{proof}
(a) \implies (b): We have that $\ini(J\cap E + J\cap E') = \ini(J\cap(E+E'))=\ini(J)\cap \ini(E+E')$, since the pair $(J,E+E')$ is $G$-nice.

(b) \implies (a): We have that 
$$
\ini(J)\cap \ini(E + E') = \ini(J \cap E + J \cap E') \subseteq \ini(J\cap (E+E')) \subseteq \ini(J)\cap \ini(E+E'),
$$
hence the inequalities in this chain become equalities. 
It follows that $\ini(J\cap (E+E'))=\ini(J) \cap \ini(E+E')$ and thus, by Theorem \ref{thm:main}(d), the pair $(J,E+E')$ is $G$-nice. 
Also, $\ini(J \cap E + J \cap E')=\ini(J\cap (E+E'))$, therefore $J \cap E + J \cap E'=J\cap (E+E')$.
\end{proof}

Given $E$ a monomial ideal, we denote by $G(E)$ its unique minimal set of monomial generators. Clearly, $G(E)\in \Gr(E)$ 
for any monomial order.

\begin{Proposition}
\label{intersections-monomials}
Let $J$ be any ideal in the polynomial ring $S$ and let $(E_i)_{i\in \Lambda}$ be a family of monomial ideals in $S$
such that the pair $(J,E_i)$ is  $G$-nice for all $i\in\Lambda$. 
We set $I_i =J+E_i$  for all $i\in \Lambda$, $I=\bigcap_{i\in\Lambda}I_i$ and $E=\bigcap_{i\in \Lambda}E_i$.
Then $(J,E)$ is a $G$-nice pair and $I=J+E$.

Also, if $\mathcal G_J\in \Gr(J)$ then $\mathcal G_J \cup G(E)$ is a Gr\"obner basis of $I$.
\end{Proposition}
 
\begin{proof}
Since the $E_i$'s are monomial ideals  then $E$ is  a monomial ideal, too. 
We have $\ini(E_i)=E_i$  for all $i\in \Lambda$ and $\ini(E)=E$.
Obviously, $\ini(J)+E \subseteq \ini(J+E)$. On the other hand, $\ini(J+E) \subseteq \bigcap_{i\in\Lambda} \ini(J+E_i) = 
\bigcap_{i\in\Lambda}(\ini(J)+E_i) = \ini(J) + \bigcap_{i\in \Lambda} E_i = \ini(J)+E$. 
Thus, the pair $(J,E)$ is $G$-nice.

Since $J+E\subseteq I$ and $\ini(J+E)=\ini(I)$, it follows that $I=J+E$. The last assertion follows immediately.
\end{proof}

The following proposition shows that the $G$-nice property behaves well with respect to taking  sums of ideals.
For any positive integer $m$ we denote $[m]=\{1,\dots, m\}$.

\begin{Proposition}
\label{Gnice-sum}
Let $E_1,\dots, E_m$ be ideals in $S$ such that $(E_i, E_j)$ is a $G$-nice pair for all $1\leq i,j \leq m$. 
Let $X\subset [m]$. We denote $E_X=\sum_{i\in X} E_i$ and $E_{X^c}=\sum_{j\in [m]\setminus X} E_j$.
Then $(E_X, E_{X^c})$ is a $G$-nice pair of ideals.
\end{Proposition}

\begin{proof}
For all $i$ we pick a Gr\"obner basis $\mathcal G_i\in \Gr(E_i)$.
 We claim that $\mathcal G_Y = \bigcup_{i\in Y}\mathcal G_i$ is a Gr\"obner basis of $E_Y$, for any $Y \subseteq [m]$. 

If $|Y|=1$, there is nothing to prove. Assume $|Y|\geq 2$ and let $f,g\in \mathcal G_Y$. If $f,g\in G_{i}$, 
for some $i\in Y$, then $S(f,g)\rightarrow_{G_i}0$ and therefore $S(f,g)\rightarrow_{G_Y}0$. If $f\in G_i$ and $g\in G_j$, with $i\neq j$ in $Y$, then $S(f,g)\rightarrow_{G_i\cup G_j}0$, since $G_i\cup G_j$ is a Gr\"obner basis of $E_i+E_j$. It follows that $S(f,g)\rightarrow_{G_Y}0$.  Thus, $G_Y$ is a Gr\"obner basis of $E_Y$, which proves our claim.

  In follows that $\mathcal{G}_{[m]}=\mathcal G_X \cup \mathcal{G}_{X^c}$ is a Gr\"obner basis for $E_{[m]}= E_X + E_{X^c}$, and  therefore the pair $(E_X, E_{X^c})$ is  $G$-nice.
\end{proof}

\begin{Corollary}
If $J$ is any ideal and $(E_i)_{i\in \Lambda}$ is a family of monomial ideals, such that the pair $(J, E_i)$ is $G$-nice for all $i\in \Lambda$, then the pair $(J, \sum_{i\in\Lambda}E_i)$ is $G$-nice.
\end{Corollary}

\begin{proof}
Note that $\sum_{i\in \Lambda}  E_i$ can be written as the sum of finitely many terms in  the sum.
  On the other hand, any two monomial ideals form a  $G$-nice pair,  so the conclusion follows by  Proposition~\ref{Gnice-sum}.
\end{proof}

\section{Creating Gr\"obner-nice pairs}
\label{sec2}

Let $J$ be an ideal in $S$. Given any ideal $E\subset S$ such that $(J,E)$ is not a $G$-nice pair, we are interested in finding ideals $F$ in $S$, ``close" to $E$, such that
\begin{equation}
\label{eq:pair}
J+F=J+E \text{ and } (J,F) \text{  is a $G$-nice pair}.
\end{equation}

\begin{Lemma}\label{smth}
If $(J,F)$ is a $G$-nice pair with $E\supseteq F$ so that $J+E=J+F$, then  $(J,E)$ is a $G$-nice pair.
\end{Lemma}
\begin{proof}
We have $\ini(J+E)=\ini(J+F)=\ini (J)+ \ini(F)\subseteq \ini (J)+\ini(E)$, hence $\ini(J+E)=\ini(J)+\ini(E)$ and $(J,E)$ is a $G$-nice pair.
\end{proof}

Based on Lemma~\ref{smth}, for \eqref{eq:pair} we should look at ideals $F \supseteq E$.
 In general, $(J, J+E)$ is a $G$-nice pair, so in the worst case we may take $F=J+E$. But sometimes, it is also the best choice, as the following example shows.

\begin{Example}
{\em
Let $J=(x^2-y^2)$ and $E=(x^2)$. We consider on $S=K[x,y]$ the revlex order with $x>y$. Let $I=J+E$. 
Then $\ini(I)=I=(x^2, y^2)$. Let $F \supseteq E$ be an ideal with $J+ F=I$ and $\ini(I)=\ini(J) +\ini(F)$. We claim that $F=I$.

Indeed, since $y^2 \in \ini(F)$, there exists $f\in F$ with $\LT(f)=y^2$. It follows that $f=y^2+ax+by+c$, where $a,b,c \in K$. Since $y^2, f\in I$, we get  that $ax+by+c \in I$, and therefore $a=b=c=0$. Thus $y^2 \in F=I$.
}
\end{Example}

\begin{Remark}
{\em
For $J$ and $E$ ideals in $S$, assume the ideal $F$ satisfies \eqref{eq:pair} and $F\supseteq E$. In order to find an ideal $E'\subseteq S$ such that $E\subseteq E' \subseteq F$  and $(J,E')$ is a $G$-nice pair, where $E'$ is as small as possible, a natural approach is the following. Set $I=J+E$. We write
$$
\ini(I)=\ini(J)+\ini(E)+ (m_1,\dots, m_s),
$$
where $m_1,\dots, m_s$ are the monomials in $G(\ini(I))\setminus (\ini(J)+\ini(E))$. Since $\ini(I)=\ini(J+F)=\ini(J)+\ini(F)$, it follows that $m_1,\dots, m_s \in \ini(F)$. We choose $g_1,\dots, g_s  \in F$ such that  $\ini(g_i)= m_i$ for $1\leq i\leq s$.
   Let $E_1=E+(g_1,\dots, g_s)$. Then $J+E_1=J+E=I$ and 
$$
\ini(I)=\ini(J)+\ini(E)+(m_1,\dots, m_s) \subseteq \ini(J)+\ini(E_1)\subseteq \ini(J+E_1)=\ini(I).
$$
Thus $(J,E_1)$ is a $G$-nice pair of ideals. 
}
\end{Remark}
 
\medskip

The following example shows that given the ideal $I$ there does not always exist 
a minimal ideal $E'$ (eventually containing $E$) such that $(J, E')$ is a $G$-nice pair with $I=J+ E'$.

\begin{Example}\label{exm:infinite}{\em
(1) We consider the lexicographic order on $S=K[x,y]$ induced by $x>y$. Let $J=(y^3)\subset I=(y^3, x^2-y^2, xy)$. 
Then $\ini(I)=(y^3, x^2, xy)$. 
We define the sequence $(\alpha_k)_{k\geq 1}$  by $\alpha_1=5$ and $\alpha_{k+1}=3\alpha_k -4$ for $k\geq 1$.
Let $g_k=xy-y^{\alpha_k}$ and $E_k=(x^2-y^2, xy-y^{\alpha_k})$. One can easily check that $J+E_k=I$, and $\ini(E_k)=(x^2, xy, y^{2\alpha_k-1})$. Therefore, $(J, E_k)$ is a $G$-nice pair of ideals for all $k$.

Since  $xg_k-y(x^2-y^2)=y^3-xy^{\alpha_k}=y^3-xy g_k-y^{2 \alpha_k -1}$, it follows that $y^3-y^{2\alpha_k-1} \in E_k$. 
Then $g_{k+1}=g_k+y^{\alpha_k}-y^{3\alpha_k -4}=g_k+y^{\alpha_k-3}(y^3-y^{2\alpha_k-1})$. Hence $E_k \subsetneq E_{k+1}$ for all $k$. We also note that $\sect_{k\geq 1} E_k =(x^2-y^2)$, $(J, (x^2-y^2))$ is a $G$-nice pair of ideals, and $J+(x^2-y^2)\subsetneq I$.

(2) In $S=K[x,y,z]$ ordered lexicographically (with $x>y>x$) we let $J=(x^2-y^2, z^2)$, $E=(xy)$, $F=(xy, y^3, z^2)$. Denote   $I=J+E=J+F=(x^2-y^2, z^2, xy)$. We note that $\ini(I)=(x^2, xy, y^3, z^2)=\ini (J)+\ini(F)$. Therefore, $(J,E)$ is not a $G$-nice pair, while $(J,F)$  is one.  We set $E_k=(xy, y^3+z^{k+1}+z^k+\dots+ z^2, z^{k+2})$. Then $E\subsetneq E_{k+1} \subsetneq E_k \subsetneq F$, while $(J, E_k)$  is a $G$-nice pair for all $k\geq 1$.
}
\end{Example}

We recall the definition of a normal form, see \cite{sing-book}.
\begin{Definition}
Consider the ideal $J\subset S$ and $\mathcal G_J\in \Gr(J)$. A \emph{normal form} with respect to $\mathcal G_J$ is a map $NF(-|\mathcal G_J):S \rightarrow S$, which satisfies the following conditions:
\begin{enumerate}
\item[(i)] $NF(0|\mathcal G_J)=0$;
\item[(ii)] if $NF(g|\mathcal G_J)\neq 0$ then $\ini(NF(g|\mathcal G_J))\notin \ini(J)$;
\item[(iii)] $g-NF(g|\mathcal G_J)=\sum_{f\in \mathcal G_J} c_f f$, where $c_f\in S$ and $\ini(g)\geq \ini(c_f f)$, for all $f\in\mathcal G_J$ with $c_f\neq 0$.
\end{enumerate}
Moreover, $NF$ is called a \emph{reduced normal form}, if for any $f\in S$ no monomial of $NF(f|\mathcal G_J)$ is contained in $\ini(\mathcal G_J)$.
\end{Definition}

Given the ideals $J,E$ in $S$, $\mathcal G_J\in \Gr(J)$ and  $NF(-|\mathcal G_J)$  a normal form with respect to $\mathcal G_J$,
we denote $NF(E|\mathcal G_J)=(NF(g|\mathcal G_J):\; g\in E)$.
  
\begin{Proposition}
With notation as above, we have:
\begin{enumerate}
\item [(i)] $J+E=J+NF(E|\mathcal G_J)$;
\item [(ii)] $(J, NF(E|\mathcal G_J))$ is a $G$-nice pair of ideals;
\item [(iii)] if $NF$ is a reduced normal form and $E=NF(E|\mathcal G_J)$, then any $h\in J+E$ can be written as $h=f+g$, such that $f\in J$, $g\in E$ and no monomial of $g$ is contained in $\ini(J)$;
\item [(iv)] if $(E_i)_{i\in\Lambda}$ is a family of ideals with $E_i=NF(E_i|\mathcal G_J)$ for all $i \in 
\Lambda$, then 
{$$NF\left(\bigcap_{i\in \Lambda}E_i|\mathcal G_J \right) \subseteq   \bigcap_{i\in\Lambda}E_i.$$}
\end{enumerate}
\end{Proposition}

\begin{proof}
$(i):$ If $h\in J+E$, then $h=NF(h|\mathcal G_J)+\sum_{f\in \mathcal G_J}c_f f$, and therefore $NF(h|\mathcal G_J)\in J+E$.

$(ii):$ Since, by $(i)$, $J+E=J+NF(E|\mathcal G_J)$, it is enough to consider the case $E=NF(E|\mathcal G_J)$ and to prove that
$\ini(J+E)\subseteq \ini(J)+ \ini(E)$.
Let $h\in J+E$. Then $h=j+g$ where $j\in J$ and $g\in E$. We also write $g=j_1+g_1$, where we let $g_1=NF(g|\mathcal{G}_J) \in NF(E|\mathcal{G}_J)=E$. Thus $h=(j+j_1)+ g_1$. If $g_1=0$, then $\ini(h)\in \ini J$. Otherwise, by the definition of the normal form, $\ini(g)\notin \ini(J)$ which implies $\ini(h) \in \ini(J)+ \ini(E)$.
 
$(iii):$ For any $h\in J+E$, the decomposition $h=(h-NF(h|\mathcal G_J))+ NF(h|\mathcal G_J)$ satisfies the required condition.

$(iv)$ is straightforward.

\end{proof}

Given the ideals $J,E\subseteq S$ we introduce the sets
\begin{eqnarray*}
\mathcal{E}_{J,E} &=&\{ F\subseteq S: E\subseteq F, J+E=J+F, (J,F) \text{ is a $G$-nice pair of ideals}\},\\
\mathcal{E}_{J,E}^{m} &=& \{ F\in \mathcal{E}_{J,E}: F \text{ is a monomial ideal}\}
\end{eqnarray*}

The previous discussion shows that the set $\mathcal{E}_{J,E}$ may not have a minimal element.
However, when $E$ is a monomial ideal and $\mathcal{E}_{J,E}^{m}\neq \emptyset$, then the latter set has a minimum.

\begin{Definition}\label{def:ehat}
Let $J$   be an ideal in $S$, and $E$  a monomial ideal in $S$.
The $G$-nice monomial closure of $E$ with respect to $J$ is the (monomial) ideal
$$
\widehat{E}= \bigcap_{E\subseteq F, \ F \text{ monomial ideal},\  (J,F) \text{ is } G\text{-nice}} F.
$$
\end{Definition} 

The ideal $\widehat{E}$ naturally depends on the ideal $J$, although this is not reflected in the notation. We prefer not to complicate the notation since it will be clear from the context what $J$ is.

\begin{Proposition}\label{ehat}
 Let $J$  be any ideal in $S$, and $E$  a monomial ideal in $S$.
Then $(J,\widehat{E})$ is a $G$-nice pair. Moreover, if $\mathcal{E}_{J,E}^{m}\neq \emptyset$, then $\widehat{E}$ is the smallest element in $\mathcal{E}_{J,E}^{m}$ with respect to inclusion.
\end{Proposition}

\begin{proof}
The first part follows from Proposition~\ref{intersections-monomials}.  For the second assertion, note that if $F\in \mathcal{E}_{J,E}^{m}$ then  $J+E\subseteq J+\widehat{E} \subseteq J+F$, hence $J+E=J+\widehat{E}$, $\widehat{E} \in \mathcal{E}_{J,E}^{m}$  and it is its smallest element.
\end{proof}

\begin{Corollary}
\label{binomial-hat}
With notation as in Proposition~\ref{ehat}, if $J$ is  a  binomial ideal, then $J+E=J+\widehat{E}$.
\end{Corollary}
\begin{proof}
Note that if $b$ is any binomial in $S$ and $m$ is any monomial in $S$, then their $S$-polynomial $S(b,m)$ is a monomial. Also, since $J$ is a binomial ideal, it has a Gr\"obner basis $\G_J$ consisting of binomials. Therefore, we can define the monomial ideal $F$  which extends $E$ by adding the monomials $S(b,m)$ where $m\in G(E)$ and $b \in \G_J$.  Then $F\in \mathcal{E}_{J,E}^{m}$ , and we apply Proposition~\ref{ehat}.
\end{proof}

The ideal $\widehat{E}$ can be computed as follows.

\begin{Remark}
{\em Let $J\subset S$ be an ideal, and let $E\subset S$ be a monomial ideal. 
Let $F\subset S$ be any  monomial ideal such that $(J,F)$ is a $G$-nice pair and $E\subseteq F$.  
Let $G_0$ be the minimal monomial generators of $\ini(J+E)$ which are not in $\ini(J)$ nor in $\ini(E)$. Clearly, $G_0\subset F$.
We let $E_1=E+(G_0)$. If $(J, E_1)$ is $G$-nice, then $\widehat{E}=E_1$. Else, we argue as above and we get a chain of monomial ideals 
$E_1\subseteq E_2\subseteq \cdots \subseteq F$. By notherianity, this chain stabilizes at some point $E_i=E_{i+1}=\cdots $ and we get $\widehat{E}=E_i$.}
\end{Remark}

\begin{Proposition}
Let $J$ be a binomial ideal and let $(E_i)_{i\in \Lambda}$ be a family of monomial ideals. Assume  $F_i \supseteq E_i$ are  monomial ideals such that $(J, F_i)$ is a $G$-nice pair and $J+E_i=J+F_i$, for all $i\in\Lambda$. 
Then  
$$
\bigcap_{i\in \Lambda} (J+ E_i) = J+\bigcap_{i \in \Lambda} \widehat{E}_i = J+\bigcap_{i\in \Lambda}{F_i}.
$$
\end{Proposition}

\begin{proof}
Using the Corollary~\ref{binomial-hat} we have that $J+E_i=J+\widehat{E}_i=F+F_i$ for all $i\in \Lambda$, hence
$$
\bigcap_{i\in \Lambda} (J+ E_i)= \bigcap _{i\in \Lambda} (J+\widehat{E}_i)=\bigcap_{i \in \Lambda} (J+ F_i).
$$
On the other hand, by  Proposition~\ref{ehat},  $(J, \widehat{E}_i)$ is a $G$-nice pair for all $i$. Now using 
Proposition~\ref{intersections-monomials} we get that $\bigcap _{i\in \Lambda} (J+\widehat{E}_i)= J+\bigcap_{i\in \Lambda}\widehat{ E}_i$ and  $\bigcap _{i\in \Lambda} (J+F_i)= J+\bigcap_{i\in \Lambda} F_i$.
\end{proof}

\begin{Example}
{\em
We consider the revlex order with $x>y>z$ on $S=K[x,y,z]$. Let $J=(x^2+y^2+z^2)$ and $E=(xy)$ be ideals in $S$.  
Note that $\mathcal G=\{x^2+y^2+z^2,  xy,  y^3+yz^2\}$ is a Gr\"obner basis of $I=J+E$. Therefore, $\ini(I)=(x^2,xy,y^3)$ strictly includes $\ini(J)+\ini(E)=(x^2,xy)$, and the pair $(J,E)$ is not $G$-nice. 

Let $F\subset S$ be any monomial ideal such that the pair $(J,F)$ is $G$-nice and $E\subseteq F$.
Since $\ini(J+E)\subseteq \ini(J+F)=(x^2)+F$, it follows that $(xy,y^3)\subset F$. Let $E_1=(xy,y^3)$. We have $(x^2,xy,y^3,yz^2)=\ini(J+E')\subseteq (x^2)+F$. Thus $yz^2\in F$.  Clearly, $E_2=(xy,y^3,yz^2)\subseteq F$. 
Since $(J, E_2)$ is a G-nice pair, we conclude that  $E_2 =\widehat{E}$.
} 
\end{Example}

\section{A special class of Gr\"obner-nice pairs of ideals}
\label{sec3}

To verify  if a set is a Gr\"obner basis implies computing the $S$-polyonomial of any two elements in the set and testing if it reduces to zero with respect to the given   set, see \cite{EH, Eis}. Inspired by this, we propose the following.

\begin{Definition}
Let $J,E$ be ideals in $S$ and $\G_J \in \Gr(J)$. We say that $E$ is $S$-nice with respect to $\G_J$ if for any $f\in \G_J$ and $g\in E$ we have $S(f,g)\in E$.
\end{Definition}
 
\begin{Example}\label{ex-snice}{\em
If $J\subseteq E$, then $E$ is $S$-nice with respect to $\G_J$ for any $\G_J\in \Gr(J)$.
}
\end{Example}

\begin{Proposition}
Assume the ideal $E$ is $S$-nice with respect to $\G_J \in \Gr(J)$. Then $(J,E)$ is a $G$-nice pair of ideals.
\end{Proposition}

\begin{proof}
Let $\G_E$ be any Gr\"obner basis for $E$. We claim that $\G_E\cup \G_J$ is a Gr\"obner basis for $E+J$.
Indeed, we only need to consider $S$-polynomials $S(f,g)$ where $f\in \G_J$ and $g\in \G_E$. Since $S(f,g)\in E$ we infer that the former reduces to $0$ w.r.t. $\G_J\cup \G_E$. Applying Theorem~\ref{thm:main}(c) finishes the proof.
\end{proof}

\begin{Proposition}
\label{snice}
Let $J,E$ be ideals in $S$ and $\G_J\in \Gr(J)$. The following statements are equivalent:
\begin{enumerate}
\item[(a)] the ideal $E$ is $S$-nice with respect to $\G_J$;
\item[(b)] for any $\G_E \in \Gr(E)$, for any $f\in \G_J$ and $g\in \G_E$ one has that $S(f,g)\in E$;
\item[(c)] there exists $\G_E\in \Gr(E)$ such that for any $f\in \G_J$ and $g\in \G_E$ one has that $S(f,g)\in E$.
\end{enumerate}
\end{Proposition}

\begin{proof}
The implications (a)\implies (b) \implies (c) are clear. We suppose (c) holds.
Without loss of generality, we may also assume that the polynomials in $\G_J$ and $\G_E$ are monic.
Let $f\in \G_J$ and $g\in \G_E$. We can write $g=\sum_{i=1}^p u_i g_i$ where $g_i \in \G_E$ and $\ini(g)\geq \ini(u_i g_i)$ for $i=1,\dots, p$ and such that $\LT (g)=\LT(u_1 g_1)$. We set $h=g- u_1 g_1$.
Then {\small
\begin{eqnarray*}
S(f,g) &=& \frac{\lcm(\ini(f), \ini(g))}{\ini(f)} f- \frac{\lcm(\ini(f), \ini(g))}{\ini(g)} g  \\
&=& \frac{\lcm(\ini(f), \ini(g))}{\ini(f)} (f-\LT(f)) -  \frac{\lcm(\ini(f), \ini(g))}{\ini(g)} (g-\LT(u_1g_1))  \\
&=&  \frac{\lcm(\ini(f), \ini(g))}{\ini(f)} (f-\LT(f))   -  \frac{\lcm(\ini(f), \ini(g))}{\ini(g)} (u_1 g_1-\LT(u_1 g_1)) \\
&   &-  \frac{\lcm(\ini(f), \ini(g))}{\ini(g)} h.
\end{eqnarray*}  }
Note that $u_1 g_1-\LT(u_1 g_1)=LT(u_1) (g_1-\LT(g_1))+ (u_1-LT(u_1))g_1$. 
Thus,{\small
\begin{eqnarray*}
S(f,g)&=&   \frac{\lcm(\ini(f), \ini(g))}{\ini(f)} (f-\LT(f))-  \frac{\lcm(\ini(f), \ini(g))}{\ini(g_1)} (g_1-\LT(g_1))  \\
         & &  -\frac{\lcm(\ini(f), \ini(g))}{\ini(g)} (  (u_1-LT(u_1))g_1+h ) \\
         &=& \frac{\lcm(\ini(f), \ini(g))}{ \lcm( \ini(f), \ini(g_1))} S(f, g_1) - \frac{\lcm(\ini(f), \ini(g))}{\ini(g)} (  (u_1-LT(u_1))g_1+h ).
\end{eqnarray*} }
Since $S(f, g_1), g_1, h \in E$ we obtain that $S(f,g) \in E$, too.
This proves statement (a).
\end{proof}

The $S$-nice property is stable when taking intersections.

\begin{Proposition}\label{S-intersect}
Let $J$ be an ideal in $S$ and $\G_J \in \Gr(J)$. Assume that in the family of ideals $(E_i)_{i \in \Lambda}$ each is $S$-nice with respect to $\G_J$. Then
\begin{enumerate}
\item[(a)] the ideal  $\bigcap_{i\in \Lambda} E_i$ is $S$-nice with respect to $\G_J$;
\item[(b)] if $(E_i, E_j)$ is a $G$-nice pair for all $i,j \in \Lambda$, then $\sum_{i\in \Lambda} E_i$ is $S$-nice with respect to $\G_J$.
\end{enumerate}
\end{Proposition}

\begin{proof}
(a): Let $f\in \G_J$ and $g\in \bigcap_{i\in \Lambda} E_i$
Since $E_i$ is $S$-nice w.r.t. $\G_J$ we get that $S(f, g)\in E_i$ for all $i\in \Lambda$. This proves (a).

(b): 
For all $i\in \Lambda$ we pick $\G_i\in  \Gr(E_i)$. Arguing as in the proof of Proposition~\ref{Gnice-sum} we get that $\G=\bigcup_{i\in \Lambda} \G_i$ is a Gr\"obner basis for $\sum_{i\in \Lambda} E_i$. Then for any $f\in \G_J$ and any $g\in \G$ we have that $S(f,g) \in \sum_{i\in \Lambda} E_i$. Conclusion follows by Proposition~\ref{snice}.
\end{proof}

An immediate consequence of the previous result is the following related form of Proposition~\ref{prop:distributivity}.

\begin{Corollary}
Let $J,E$ and $E'$ be ideals in $S$ and $\G_J\in \Gr(J)$. Assume that  $E$ and $E'$ are $S$-nice with respect to $\G_J$.
Then the following conditions are equivalent:
\begin{enumerate}
\item[(a)] $(J+E)\cap (J+E')=J + (E\cap E')$;
\item[(b)] $\ini((J+E)\cap (J+E')) = \ini(J) + \ini(E\cap E')$.
\end{enumerate}
\end{Corollary}

In view of Proposition~\ref{S-intersect}, for any ideal we can define its  $S$-nice (monomial) closure.

\begin{Definition}
Let $J$ be any ideal in $S$ and let $\G_J \in \Gr(J)$. For any ideal $E\subset S$ we set
$$
\widetilde{E}=\bigcap_{E\subseteq F,\  F \text{ is } S\text{-nice  w.r.t. }\G_J} F
$$
and we call it the $S$-nice closure of $E$ with respect to $\G_J$.

Moreover, if $E$ is a monomial ideal, we set
$$
{E}^\sharp=\bigcap_{E\subseteq F,\  F \text{ monomial ideal  is } S\text{-nice  w.r.t. }\G_J} F
$$
and we call it the $S$-nice monomial closure of $E$ with respect to $\G_J$.
\end{Definition}

As with Definition~\ref{def:ehat}, we prefer not to complicate notation and include $\G_J$ in it, as it will be clear from the context the Gr\"obner basis which is used.

\begin{Remark}{\em 
By Proposition~\ref{S-intersect}, $\widetilde{E}$ (resp. $E^\sharp$) is indeed the smallest ideal (resp. the smallest monomial ideal) in $S$ which is $S$-nice with respect to $\G_J$. Clearly, $E\subseteq \widetilde{E} \subseteq E^\sharp$, and also $\widehat{E} \subseteq E^\sharp$. 
By Example~\ref{ex-snice}, the ideal $J+E$ is $S$-nice w.r.t. $\G_J$, hence $J+E\subseteq J+\widetilde{E} \subseteq J+E$. The latter  implies that
$$
J+E=J+\widetilde{E}. 
$$}
\end{Remark}

\begin{Example}
\label{snice-exam}{\em
In $S=K[x,y]$ we consider $J=(x^2+y^2)$ with $\G_J=\{x^2+y^2\}$ and $E=(x^2)$.  If $x>y$ then $S(x^2+y^2, x^2)=y^2$, and so $\widetilde{E}=(x^2, y^2)=J+E$.
On the other hand, if $y>x$ then $S(x^2+y^2, x^2)=x^4 \in E$ and thus $\widetilde{E}=E=(x^2)$.
}
\end{Example}

In general, the following proposition is useful for computing $\widetilde{E}$ and $E^\sharp$.

\begin{Proposition} 
\label{inductive}
Let $J, E$ be  ideals in $S$ and $\G_J\in \Gr(J)$.  Then
\begin{enumerate}
\item[(a)] $\widetilde{E} =\sum_{i\geq 0} E_i$, where the ideal $E_i$ is defined inductively as $E_0=E$ and $E_{i+1}=E_i+(S(f,g):f\in \G_J, g\in E_i)$ for all $i>0$.
\item[(b)] If $E$ is a monomial ideal, then $E^\sharp= \sum_{i\geq 0} F_i$, where  $F_0=E$ and $F_{i+i}$ is the ideal generated by  the monomial terms of the polynomials in $\widetilde{F}_i$, for all $i>0$.
\end{enumerate}
\end{Proposition}

\begin{proof}
(a): If $f\in \G_J$ and $g\in G_i$, then $S(f,g) \in E_{i+1}$. Therefore, $\sum_{i\geq 0} E_i$ is $S$-nice with respect to $\G_J$.
Conversely, since $E_0\subset \widetilde{E}$ and $\widetilde{E}$ is $S$-nice w.r.t. $\G_J$, it follows that $S(f,g)\in \widetilde{E}$ for any $f\in \G_J$ and $g\in E$. Therefore, $E_1\subseteq \widetilde{E}$.  Inductively, we get that $E_i \subseteq \widetilde{E}$ for any $i \geq 0$. This completes the proof.

(b): If  $m\in F_i$ is  a monomial and $f\in \G_J$, then $S(f, m)\in \widetilde{F_i}\subseteq F_{i+1}$. Therefore, $\sum_{i\geq 0} F_i$ is $S$-nice with respect to $\G_J$. Conversely, if $m \in E$ is a monomial and $f\in \G_J$, then $S(f,m) \in E^\sharp$.
Since $E^\sharp$ is a monomial ideal,, any monomial which is in the support of $S(f,m)$ is in $E^\sharp$. Therefore, $F_1\subseteq E^\sharp$.
Inductively, we get $F_i \subseteq E^\sharp$ for all $i\geq 0$. Thus $\sum_{i\geq 0} F_i=E^\sharp$.
\end{proof}

\begin{Example}{\em
In $S=K[x,y,z]$ we consider the revlex order with $x>y>z$ and the ideals $J=(x^2+y^2+z^2)$, $E=(xy)$, $E_1=(xy, y^3+yz^2)$ and $E_2=(xy, y^3+yz^2+x^2+y^2+z^2)$. From Example~\ref{exm} we have that $J+E_1=J+E_2=J+E$, and that   $(J, E_1)$, $(J, E_2)$ are $G$-nice pairs of ideals. Let $\G_J=\{x^2+y^2+z^2\} \in \Gr(J)$. We claim that
$E_1$ is $S$-nice with respect to $\G_J$, but $E_2$ is not.

We note that $\G_1=\{xy, y^3+yz^2\} \in \Gr(E_1)$ and $\G_2=\{xy, y^3+yz^2+x^2+y^2+z^2, x^3+xz^2\} \in \Gr(E_2)$.
Also, $S(x^2+y^2+z^2, xy)=y^3+yz^2 \in E_1$ and $S(x^2+y^2+z^2, y^3+yz^2)=y^3(y^2+z^2)-x^2(y^3+yz^2)= y^2(y^3+yz^2)-xy(xy^2+xz^2)\in E_1$. This proves the claim. 
Moreover, since $S(x^2+y^2+z^2, xy)=y^3+yz^2$ we infer that, when computed with respect to $\G_J$,  $\widetilde{E}=E_1$.

Let $U$ be a monomial ideal in $S$ which is $S$-nice with respect to $\G_J$, and $E\subseteq U$.
Since $S(x^2+y^2+z^2, xy)=y^3+yz^2 \in U$ one has  that $y^3, yz^2 \in U$. Let  $L=(xy, y^3, yz^2)$.
As $S(x^2+y^2+z^2, y^3)=y^5+y^3z^2\in L$ and $S(x^2+y^2+z^2, yz^2)=y^3z^2+yz^4 \in L$, it follows that $L$ is $S$-nice with respect to $\G_J$ and moreover, $E^\sharp= L$.  We also remark that $J+E\subsetneq J+E^\sharp$.
}
\end{Example}

\begin{Proposition} \label{proplast}
Let $J$ be an ideal in $S$, $\G_J \in \Gr(J)$ and $E$ a monomial ideal in $S$. Then
\begin{enumerate}
\item[(a)] if there exists a monomial ideal $F$ in $S$ which is $S$-nice with respect to $\G_J$ and $J+E=J+F$, then $J+E^\sharp=J+E$;
\item[(b)] if $J$ is a binomial ideal and $\G_J \in \Gr(J)$ consists of binomials, then $\widetilde{E}=E^\sharp$.
\end{enumerate}
\end{Proposition}

\begin{proof}
Part (a) is clear. For (b) we let $E_0=E$ and we note that $S(b,m)$ is a monomial, for any binomial $b$ and monomial $m$.
Therefore, using the notation from Proposition	~\ref{inductive} we obtain an ascending chain of monomial ideals $E_{i+1}=E_i+(S(f,g):f\in \G_J, g\in E_i)$. This shows that $\widetilde{E}=\sum_{i\geq 0} E_i$ is a monomial ideal and we are done.
\end{proof}

\medskip

\noindent{\bf Acknowledgement}.  
We thank Aldo Conca for pointing our attention to  his paper (\cite{conca}).
We gratefully acknowledge the use of the computer algebra system Singular (\cite{Sing}) for our experiments.

The authors were partly supported by a grant  of the Romanian Ministry of Education, CNCS--UEFISCDI under the 
project  PN-II-ID-PCE-2011-3--1023. 

 \medskip
\noindent The paper is in final form and  no similar paper has been or is being submitted elsewhere.

\medskip

\noindent{\bf Note added in the proof}.  After this paper was finished, we learned from Matteo Varbaro that the equivalence of conditions  (a) and (d) in Theorem~\ref{thm:main} has also  been proved in the Ph.D. thesis of  Michela Di Marca, {\em Connectedness properties of dual graphs of standard graded algebras}, University of Genova, December 2017.

{}


\begin{thebibliography}{}
 
  
\bibitem{conca} A.\ Conca, \textit{Gorenstein ladder determinantal rings}, J. London Math. Soc. \textbf{(2) 54} (1996), no. 3, 453--474.
	
\bibitem{Sing} W.\ Decker, G.-M.\ Greuel, G.\ Pfister, H.\ Sch{\"o}nemann,
\newblock {\sc Singular} {4-1-2} --- {A} computer algebra system for polynomial computations.
\newblock {http://www.singular.uni-kl.de} (2019).

 
\bibitem{Eis} D.\ Eisenbud, \textit{Commutative Algebra with a View Toward Algebraic Geometry}, Graduate Texts in Mathematics, vol. {\bf 150}, Springer, 1995. 

\bibitem{EH} V.\ Ene, J.\ Herzog, \textit{Gr\"obner Bases in Commutative Algebra}, Graduate Studies in Mathematics, vol. {\bf 130}, American Mathematical Society, 2012.
 
 
\bibitem{sing-book} G.-M.\ Greuel, G.\  Pfister, \textit{A  Singular Introduction to Commutative Algebra}, 2nd ed.,  Springer, 2008. 

\bibitem{HoSu} S.~Ho\c sten, S.~Sullivant, \textit{Ideals of adjacent minors}, J. Algebra {\bf 277} (2004), 615--642.

\bibitem{MohRauh} F.~Mohammadi, J.~Rauh, \textit{Prime splittings of determinantal ideals}, Comm. Algebra {\bf 46} (2018), 2278--2296.

 
\end{thebibliography}
\end{document}